\documentclass[10pt,a4paper]{article}
\include{latin}

\usepackage{nomencl}

\usepackage[square,sort,comma,numbers]{natbib}
\usepackage{amsmath}
\usepackage{amsfonts,amscd}
\usepackage[all]{xy}
\usepackage{amssymb}
\usepackage{makeidx}
\usepackage{graphicx}
\usepackage{tikz}
\usetikzlibrary{matrix}
\usepackage{framed}
\usepackage{wasysym}

\usepackage[...]{youngtab}

\usepackage{geometry}
\usepackage{amsthm}

\newtheorem{theorem}{Theorem}[section]
\newtheorem{lemma}[theorem]{Lemma}
\newtheorem{proposition}[theorem]{Proposition}
\newtheorem{corollary}[theorem]{Corollary}

\begin{document}
\title{Congruence subgroups of braid groups}
\author{Charalampos Stylianakis}
\date{ }
\maketitle

\begin{abstract}
	In this paper we give a description of the generators of the prime level congruence subgroups of braid groups. Also, we give a new presentation of the symplectic group over a finite field, and we calculate symmetric quotients of the prime level congruence subgroups of braid groups. Finally, we find a finite generating set for the level-3 congruence subgroup of the braid group on 3 strands.
\end{abstract}

\section{Introduction}

Let $B_n$ be the braid group on $n$ strands. By evaluating the (unreduced) Burau representation $B_n \to \mathrm{GL}_{n-1}(\mathbb{Z}[t^{\pm 1}])$ at $t=-1$ we obtain a symplectic representation
\[
\rho : B_{n} \rightarrow
\begin{cases} 
\hfill \mathrm{Sp}_{n-1}(\mathbb{Z})    \hfill & \text{ if $n$ is odd,} \\
\hfill (\mathrm{Sp}_{n}(\mathbb{Z}))_{u} \hfill & \text{ if $n$ is even,}
\end{cases}
\]
where $(\mathrm{Sp}_{n}(\mathbb{Z}))_{u}$ is the subgroup of $\mathrm{Sp}_{n}(\mathbb{Z})$ fixing one vector $u \in \mathbb{Z}^{n}$ \cite[Proposition 2.1]{GJ1} (see also \cite{BM} and \cite{C1}).

For a positive integer $m$, the projection $\mathbb{Z} \to \mathbb{Z}/m$ induces a representation as follows:
\[\rho_m : B_{n} \rightarrow
\begin{cases} 
\hfill \mathrm{Sp}_{n-1}(\mathbb{Z}/m)    \hfill & \text{ if $n$ is odd,} \\
\hfill (\mathrm{Sp}_{n}(\mathbb{Z}/m))_{u} \hfill & \text{ if $n$ is even.}
\end{cases}
\]
Note that if $m=1$, then $\rho_1=\rho$. For $i>1$ the kernel of $\rho_m$ is denoted by $B_n[m]$ and it is called the \emph{level}-\emph{m congruence subgroup of} $B_n$. The kernel of $\rho$ is called the \emph{braid Torelli} group, and it is denoted by $\mathcal{BI}_n$. The group $\mathcal{BI}_n$ has been extensively studied by Hain \cite{RH}, Brendle-Margalit \cite{BM1,BM2}, and Brendle-Margalit-Putman \cite{BMP}.

For $p$ prime, A'Campo proved that the homomorphism $\rho_p$ is surjective, by explicitly calculating the image of $\rho_p$ \cite[Theorem 1 (1)]{C1}. Wanjryb gave a presentation of $\mathrm{Sp}_{n-1}(\mathbb{Z}/p)$ and $(\mathrm{Sp}_{n}(\mathbb{Z}/p))_{u}$ as quotients of $B_n$ \cite[Theorem 1]{W}. Let $PB_n$ be the \emph{pure braid group}, that is, the kernel of the epimorphism $B_n \to S_n$, where $S_n$ is the symmetric group on $n$ letters. Our first result is an analogue of Wanjryb's theorem.

\paragraph{Theorem A} \textit{For $p$ prime, the groups $\mathrm{Sp}_{n-1}(\mathbb{Z}/p)$ and $(\mathrm{Sp}_{n}(\mathbb{Z}/p))_{u}$ admit a presentation as quotients of the pure braid group $PB_n$.}\\

\begin{flushleft}
This result is given as Theorem \ref{SYPRE} in the paper. 
\end{flushleft}

A result of Arnol'd shows that $B_n[2]=PB_n$, where $PB_n$ is the pure braid group  \cite{A1}. Therefore, for every $k$ even, we have that $B_n[k] \unlhd PB_n$. Our second result extends A'Campo's theorem.

\paragraph{Theorem B} \textit{For $m=2p_1...p_k$, where $p_i\geq3$ are primes, we have that $PB_n/ B_n[m]$ is isomorphic to $\bigoplus\limits_{i=1}^k \mathrm{Sp}_{n-1}(\mathbb{Z}/p_i)$ if $n$ is odd, and $\bigoplus\limits_{i=1}^k (\mathrm{Sp}_{n}(\mathbb{Z}/p_i))_{u}$ if $n$ is even.}\\

\begin{flushleft}
Theorem B is Theorem \ref{PBQ1} (see also Theorem \ref{PBQ2}) in the paper.
\end{flushleft} 

We also characterize quotient groups of congruence subgroups of braid groups. The braid group $B_n$ surjects onto the symmetric group $S_n$. The kernel of this map is well known to be the pure braid group $PB_n$. Also, by a result established by A'rnold \cite{A1} the group $PB_n$ is isomorphic to $B_n[2]$. See also \cite[Section 2]{BM} for further discussion. Therefore, we have $B_n / B_n[2] \cong S_n$. We generalize this result as stated in the following theorem.

\paragraph{Theorem C.} For $p$ prime number, the group $B_n[p] / B_n[2p]$ is isomorphic to $S_n$.

\begin{flushleft}
Theorem C is Theorem \ref{symquo} in the paper.
\end{flushleft}

\paragraph{Topological description of congruence subgroups.} A key part of the paper is a topological interpretation of $B_n[p]$, for $p\geq3$ prime, given in Section 4. The content of Section 4 was inspired by Powell, who based on Birman's work on the presentation of the symplectic group \cite[Theorem 1]{JB1}, to show that the Torelli subgroup of the mapping class group is normally generated by bounding pair maps, and Dehn twists about separating simple closed curves \cite[Theorem 2]{JP}. 

Theorems A and B are used to find normal generators for $B_n[m]$, where $m=2p_1...p_k$ and $p_i$ is an odd prime. Motivated by Section 4 it would be interesting to find a topological description of the generators of $B_n[m]$ in the future.

\paragraph{Related results.} The mapping class group $\mathrm{Mod}(\Sigma)$ of an orientable surface $\Sigma$ is the group of isotopy classes of homeomorphisms that preserve the orientation of $\Sigma$, fix the boundary pointwise, and preserve the set of marked points setwise. We denote by $T_c$ a Dehn twist about a simple closed curve $c$. Let $\Sigma^b_g$ be a surface of genus $g\geq1$ with $b$ boundary components, where $b\in \{1,2\}$. It is a special case of theorem of Birman-Hilden \cite{BH} that $B_{2g+b}$ embeds into $\mathrm{Mod}(\Sigma^b_g)$ \cite[Section 9.4]{BFM}. We denote the image of this embedding by $\mathrm{SMod}(\Sigma^b_g)$. As mentioned in the previous page, the braid Torelli $\mathcal{BI}_{2g+b}$ is the kernel of the symplectic representation of $B_{2g+b}$. Hain conjectured that $\mathcal{BI}_{2g+b}$ is isomorphic to the group generated by Dehn twists about separating simple closed curves inside $\mathrm{SMod}(\Sigma^b_g)$ \cite{RH}. This conjecture was proved by Brendle-Margalit-Putman \cite[Theorem A]{BMP}, and also studied by Brendle-Margalit \cite{BM1,BM2}. By the definitions given in the beginning of the paper, the group $\mathcal{BI}_{2g+b}$ is a subgroup of $B_{2g+b}[m]$, for any $m \in \mathbb{N}$.

For $m\geq2$, consider $B_{2g+b}[m]$ as a subgroup of $ \mathrm{SMod}(\Sigma^b_g) \cong B_{2g+b}$. A consequence of a work of Arnol'd shows that $B_{2g+b}[2]$ is isomorphic to the pure braid group $PB_{2g+b}$ \cite{A1} (see \cite[Section 2]{BM} for explanation of this isomorphism). Combining the latter result with the work of Humphries \cite[Theorem 1]{SH} we obtain that $B_{2g+b}[2]$ is isomorphic to the normal closure of a square of a Dehn twist about nonseparating simple closed curve in $\mathrm{SMod}(\Sigma^b_g)$. Brendle-Margalit extended the latter result by proving that the normal closure of the $4^{th}$ power of a Dehn twist about a nonseparating simple closed curve in $\mathrm{SMod}(\Sigma^b_g)$ is isomorphic to $B_{2g+b}[4]$ \cite[Main Theorem]{BM}. 



Let $\mathcal{T}_{2g+b}(m)$ be the normal closure of the $m^{th}$ power of a Dehn twist in $\mathrm{SMod}(\Sigma^b_g)$, where $g\geq1$ and $b=1,2$. Coxeter proved that $\mathcal{T}_{2g+b}(m)$ is a finite index subgroup of $\mathrm{SMod}(\Sigma^b_g)=B_{2g+b}$ if and only if $(2g+b-2)(m-2)<4$ \cite[Section 10]{C2}. As mentioned above, $\mathcal{T}_{2g+b}(2) = B_{2g+b}[2]$. Furthermore, Humphries gave a complete description of when a group generated by $\{\mathcal{T}_{2g+b}(m_i) \mid m_i \in \mathbb{N}\}$, for finite number of $m_i$, is of finite index in $PB_{2g+b}$ \cite[Theorem 1]{HU}. In addition, Funar-Kohno proved that the intersection of all $\mathcal{T}_{2g+b}(2m)$, where $m\in \mathbb{N}$, is trivial \cite[Theorem 1.1]{FK1}.

Finally, we note a more general definition of congruence subgroups of braid groups. Let $F_n$ be the free group of rank $n$. There is an inclusion $B_n \to \mathrm{Aut}(F_n)$ \cite[Theorem 1.9]{JB}. Consider a characteristic subgroup $H$ of finite index in $F_n$. The kernel of $\mathrm{Aut}(F_n) \to \mathrm{Aut}(F_n / H)$ is called \emph{principal congruence subgroup}, and any finite index subgroup of $\mathrm{Aut}(F_n)$ containing a principal congruence subgroup is called \emph{congruence subgroup}. A group $G$ is said to have the \emph{congruence subgroup propery} if every finite index subgroup of $G$ contains a principal congruence subgroup. Asada proved that $B_n$ satisfies the congruence subgroup property by using the notions of field extensions and profinite groups \cite[Theorem 3A, Theorem 5]{AS}. In contrast with Asada's techniques, Thurston gave a more elementary proof to the congruence subgroup property of $B_n$ \cite{BMC}.

\paragraph{Outline of the paper.} In Section 2 we give basic background on braid groups, hyperelliptic mapping class groups, the symplectic representation of braid groups, and the congruence subgroups of braid groups. In Section 3 we recall some key results about the congruence subgroups of symplectic groups. In Section 4 we give a topological interpretation of the generators of the prime level congruence subgroups of braid groups. In Section 5 we prove Theorems A and B. In Section 6 we prove Theorem C.

\paragraph{Acknowledgments.} I would like to thank my PhD supervisor Tara Brendle for her support during my work on this paper.

\section{Preliminaries}
In this section we recall the definition of braid groups, hyperelliptic  mapping class groups, and the symplectic representation of braid groups.

\subsection{Definitions of braid groups}

\begin{figure}[h]
\begin{center}
\includegraphics[scale=.2]{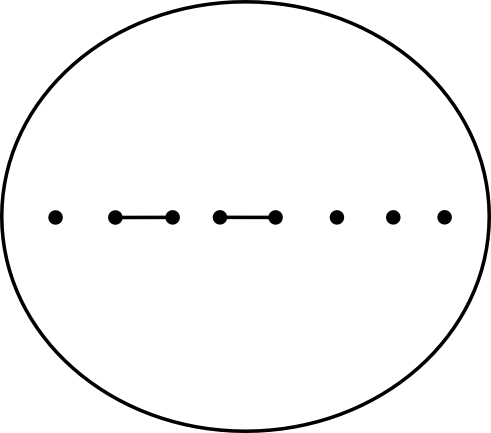}
\includegraphics[scale=.2]{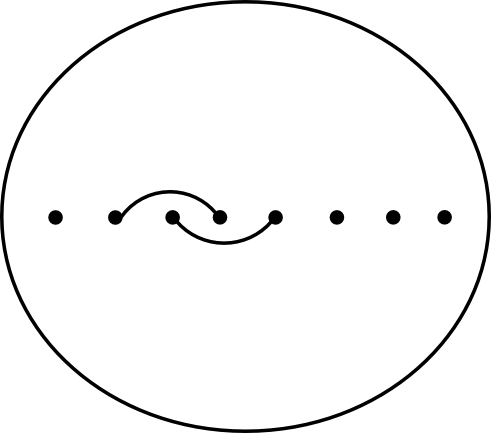}
\end{center}
\caption{The action of $\sigma_3$ on a punctured disc.}
\label{braidgroup}
\begin{picture}(22,12)

\end{picture}
\end{figure}

\paragraph{Braid groups.} For detailed description of the following definition, see Birman-Brendle's survey \cite{BB}. Let $\Sigma^b_{g,n}$ denote an orientable surface of genus $g$ with $n$ punctures and $b$ boundary components. If $n=0$ we will simply write $\Sigma^b_{g}$. If $g=0$ and $b=1$ then $\Sigma^1_{0,n}$ is homeomorphic to a punctured disc. We enumerate the punctures from left to right. The \emph{braid group} $B_n$ on $n$ strands is defined to be the mapping class group $\mathrm{Mod}(\Sigma^1_{0,n})$ of $\Sigma^1_{0,n}$. For $1\leq i \leq n-1$ we denote by $\sigma_i$ the mapping classes that interchanges the punctures $i,i+1$ as depicted in Figure \ref{braidgroup} for $i=3$. The mapping classes $\sigma_i$ are called half-twists. It turns out that $\sigma_i$ generate the braid group $B_n$. In fact we have the following presentation
\[ \left< \sigma_1,...,\sigma_{n-1} \: | \: \sigma_i \sigma_{i+1} \sigma_i=\sigma_{i+1} \sigma_i \sigma_{i+1}, \sigma_i \sigma_j = \sigma_j \sigma_i \: \mathrm{when} \: |i-j|>1  \right>. \]

Consider the symmetric group $S_n$, and for $1 \geq i \geq n-1$ let $s_i$ denote the generators of $S_n$, that is the transpositions $(i,i+1)$. The map $B_n \to S_n$ defined by $\sigma_i \mapsto s_i$ is a well defined homomorphism with kernel the \emph{pure braid group} $PB_n$. Let $1\leq i<j\leq n-1$, we denote by $a_{i,j}$ the element $\sigma_{j-1}... \sigma^2_i ... \sigma_{j-1}$. For $1\leq i<j\leq n-1$ the group $PB_n$ admits a presentation with generators $a_{i,j}$ and relations
\begin{enumerate}
\item[P1.] $\begin{aligned}[t]
a^{-1}_{r,s} a_{i,j} a_{r,s}  = a_{i,j}, \: 1 \leq r<s<i<j \leq n \: \mathrm{or} \: 1 \leq i<r<s<j \leq n,
\end{aligned}$

\item[P2.] $\begin{aligned}[t]
a^{-1}_{r,s} a_{i,j} a_{r,s}  = a_{r,j} a_{i,j} a^{-1}_{r,j} , \: 1 \leq r<s=i<j \leq n,
\end{aligned}$

\item[P3.] $\begin{aligned}[t]
a^{-1}_{r,s} a_{i,j} a_{r,s}  = (a_{i,j} a_{s,j})a_{i,j}(a_{i,j} a_{s,j})^{-1} , \: 1 \leq r=i<s<j \leq n,
\end{aligned}$

\item[P4.] $\begin{aligned}[t]
a^{-1}_{r,s} a_{i,j} a_{r,s}  = (a_{r,j} a_{s,j} a^{-1}_{r,j} a^{-1}_{s,j})a_{i,j}(a_{r,j} a_{s,j} a^{-1}_{r,j} a^{-1}_{s,j})^{-1} , \: 1 \leq r<i<s<j \leq n.
\end{aligned}$
\end{enumerate}
For more details about definitions and presentations of $B_n$ and $PB_n$ see \cite[Chapter 1]{BB}.

\begin{figure}[h]
\begin{center}
\includegraphics[scale=.5]{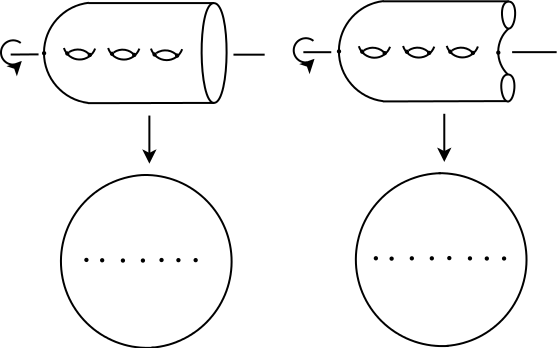}
\end{center}
\caption{Action of the hyperelliptic involution.}
\label{twofold}
\begin{picture}(22,12)

\end{picture}
\end{figure}

\paragraph{Hyperelliptic mapping class groups.} Let $c$ be a nonseparating simple closed curve on a surface $\Sigma^b_{g,n}$. We denote by $T_c$ the Dehn twist about the curve $c$. Dehn twists about nonseparating simple closed curves generate $\mathrm{Mod}(\Sigma^b_g)$. Consider a hyperelliptic involution $\iota$ as depicted in Figure \ref{twofold}. For $b=1,2$, $\iota$ acts on $\Sigma^b_g$.  Since $\iota$ does not fix the boundary components of $\Sigma^b_g$ pointwise, then $\iota \notin \mathrm{Mod}(\Sigma^b_g)$. We have a two fold branched cover $\Sigma^b_g \rightarrow \Sigma^b_g / \iota.$ Topologically $\Sigma^b_g / \iota$ is homeomorphic to $\Sigma^1_{0,2g+b}$ (see Figure \ref{twofold}). We note that if $q_1,q_2$ denote the boundary components of $\Sigma^2_g$, then $\iota(q_1)=q_2$.

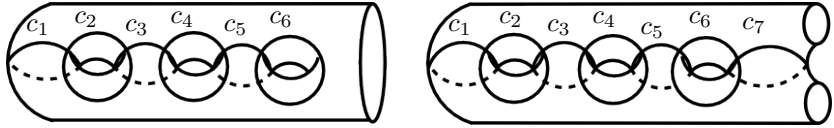
\begin{figure}[h]
\begin{center}
\includegraphics[scale=.45]{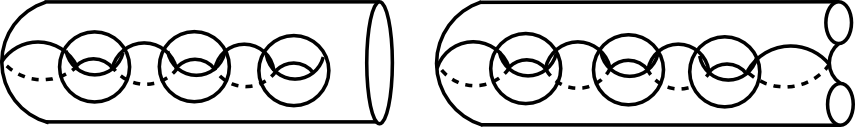}
\end{center}
\caption{Generators of the hyperelliptic mapping class group.}
\label{hyperm}
\begin{picture}(22,12)
\put(65,82){$c_1$}
\put(83,85){$c_2$}
\put(102,82){$c_3$}
\put(119,85){$c_4$}
\put(139,82){$c_5$}
\put(156,85){$c_6$}
\put(223,82){$c_1$}
\put(242,85){$c_2$}
\put(260,82){$c_3$}
\put(277,85){$c_4$}
\put(295,81){$c_5$}
\put(313,85){$c_6$}
\put(333,83){$c_7$}
\end{picture}
\end{figure}
Consider the curves $c_i$ depicted in Figure \ref{hyperm}, and let $\sigma_i$ be the generators of $B_{2g+b}$. We define a map $\xi:B_{2g+b}\rightarrow \mathrm{Mod}(\Sigma^b_g)$ by $\xi(\sigma_i) = T_{c_i}$. Since the braid, and the disjointness relations are satisfied by $\sigma_i$ and $T_{c_i}$, then $\xi$ is a homomorphism. The image of $\xi$ is called \emph{hyperelliptic mapping class group}, and it is denoted by $\mathrm{SMod}(\Sigma^b_g)$. In fact we have $B_{2g+b}\cong \mathrm{SMod}(\Sigma^b_g)$ \cite[Theorem 9.2]{BFM} (see also \cite{PV}).

\subsection{Symplectic representation}

In this section we will construct a representation for the braid group $B_n$. Firstly, we recall the definition of $\mathrm{Sp}_{2n}(\mathbb{Z})$. Let $J$ be the $2n \times 2n$ matrix
\[\left( \begin{array}{ccc}
0 & I_n\\
-I_n & 0\\
\end{array} \right).\]
The symplectic group with integer coefficients is defined to be 
\[\mathrm{Sp}_{2n}(\mathbb{Z}) = \{ A \in \mathrm{GL}(2n,\mathbb{Z}) \, \mid \, A^T J A = J \}.\]
We also define the symplectic group with coefficients in $\mathbb{Z}/m$ to be
\[\mathrm{Sp}_{2n}(\mathbb{Z}/m) = \{ A \in \mathrm{GL}(2n,\mathbb{Z}) \, \mid \, A^T J A \equiv J \: \mathrm{mod}(m) \}\]
where $m \in \mathbb{N}$. For a fixed $u \in \mathbb{Z}^{2n}$, we also recall
\[(\mathrm{Sp}_{2n}(\mathbb{Z}))_{u} = \{ t \in \mathrm{Sp}_{2n}(\mathbb{Z}) \mid t(u)=u \}.\]

Consider $g\geq1$ and $b=1,2$. Since $B_{2g+b}\cong \mathrm{SMod}(\Sigma^b_g)$, we will use the action of $\mathrm{SMod}(\Sigma^b_g)$ on the first homology of $\Sigma^b_g$ to construct a representation for $B_{2g+b}$. 

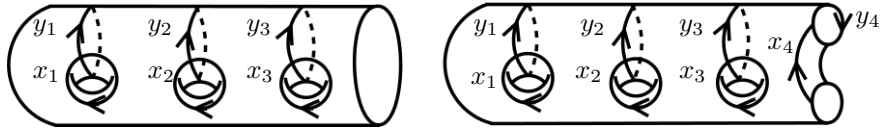
\begin{figure}[h]
	\begin{center}
		\includegraphics[scale=.45]{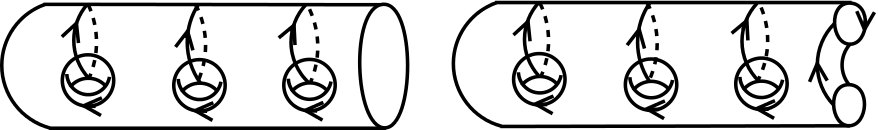}
	\end{center}
	\caption{Standard generators for $\mathrm{H}_1(\Sigma^1_g)$, and $\mathrm{H}^{P}_1(\Sigma^2_{g},\mathbb{Z})$.}
	\begin{picture}(22,12)
	\put(63,85){$y_1$}
	\put(63,67){$x_1$}
	\put(106,85){$y_2$}
	\put(106,67){$x_2$}
	\put(143,85){$y_3$}
	\put(143,67){$x_3$}
	\put(228,85){$y_1$}
	\put(227,65){$x_1$}
	\put(268,85){$y_2$}
	\put(266,67){$x_2$}
	\put(305,85){$y_3$}
	\put(305,67){$x_3$}
	\put(371,88){$y_4$}
	\put(338,78){$x_4$}
	\end{picture}
	\label{sympb1}
\end{figure}

\paragraph{Construction of the representation.} We denote by $\iota_a$ the algebraic intersection number between curves of $\Sigma^b_g$ for $g\geq1$ and $b=1,2$. The form $\iota_a$ is an alternating bilinear and nondegenerate. Every element of the mapping class group preserves $\iota_a$ \cite[Section 6.3]{BFM}. Consider $b=1$; the oriented curves $x_i,y_i$ of $\Sigma^1_g$ of Figure \ref{sympb1} form a symplectic basis for $\mathrm{H}_1(\Sigma^1_g;\mathbb{Z})$. The action of $\mathrm{SMod}(\Sigma^1_g)$ on $\mathrm{H}_1(\Sigma^1_g;\mathbb{Z})$ induces the following representation:
\[\mathrm{SMod}(\Sigma^1_g) \to \mathrm{Sp}_{2g}(\mathbb{Z}).\]

If $b=2$, the module $\mathrm{H}_1(\Sigma^2_g;\mathbb{Z})$ is not symplectic. Thus, we will consider a different module. Fix a point on each of the boundaries of $\Sigma^2_g$, and denote by $Q$ the set that contains those two points. Denote also by $P$ the set that contains the two boundary components. We set $\mathrm{H}^P_1(\Sigma^2_g;\mathbb{Z}) \cong \mathrm{H}_1(\Sigma^2_g,Q;\mathbb{Z})/\langle P \rangle$. The module $\mathrm{H}^P_1(\Sigma^2_g;\mathbb{Z})$ is symplectic \cite[Section 2.1]{BM} (see also \cite{PM1}). The basis of $\mathrm{H}^P_1(\Sigma^2_g;\mathbb{Z})$ is $x_i,y_i$ as indicated on the right hand side of Figure \ref{sympb1}. The action of $\mathrm{SMod}(\Sigma^2_g)$ on $\mathrm{H}^P_1(\Sigma^2_g;\mathbb{Z})$ induces the following representation:
\[\mathrm{SMod}(\Sigma^2_g) \to (\mathrm{Sp}_{2g+2}(\mathbb{Z}))_{y_{g+1}},\]
where $(\mathrm{Sp}_{2g+2}(\mathbb{Z}))_{y_{g+1}}$ stands for the subgroup of $\mathrm{Sp}_{2g+2}(\mathbb{Z})$ that fixes the vector $y_{g+1}$.

Since the map $\xi : B_{2g+b} \to \mathrm{SMod}(\Sigma^b_g)$ is an isomorphism, we have a well defined representation
\[
\rho : B_{2g+b} \rightarrow
\begin{cases} 
\hfill \mathrm{Sp}_{2g}(\mathbb{Z})    \hfill & \text{ if $b=1$} \\
\hfill (\mathrm{Sp}_{2g+2}(\mathbb{Z}))_{y_{g+1}} \hfill & \text{ if $b=2$.}
\end{cases}
\]

\paragraph{Image of the representation.} We denote also by $[c]$ the homology class of a curve $c$ in $\Sigma^b_g$. For $x,c$ nonseparating simple closed curves in $\Sigma^b_g$, the automorphism $T_{[c]}([x]) = [x]+\iota_a(x,c)[c]$ is called a transvection \cite[Section 6.6.3]{BFM}. We remark that for every integer $m$, we have $T^m_{[c]}([x]) = [x]+m \iota_a(x,c)[c]$. 

Let $T_{c_i}$ be a Dehn twist about a curve $c_i$ indicated in Figure \ref{hyperm}. The image of $T_{c_i}$ under the symplectic representation is the transvection $T_{[c_i]}$. Also, since $\xi(\sigma_i)=T_{c_i}$ as explained in the previous section, we have $\rho(\sigma_i)=T_{[c_i]}$. We note also that $\rho(\sigma^m_i)=T^m_{[c_i]}$.

\paragraph{Kernel of the symplectic representation.} Assume that $b=1,2$, $g\geq0$, and recall that $B_{2g+b}=\mathrm{Mod}(D_{2g+b})$. The kernel of the symplectic representation $\rho$ is denoted by $\mathcal{BI}_{2g+b}$, and it is called the \emph{braid Torelli}. It is a result by Brendle-Margalit-Putman that $\mathcal{BI}_{2g+b}$ is generated by Dehn twists about simple closed curves surrounding 3 or 5 number of puncture points \cite[Theorem C]{BMP}.

Consider the isomorphism $\xi : B_{2g+b} \to \mathrm{SMod}(\Sigma^b_g)$. The image of $\mathcal{BI}_{2g+b}$ in $\mathrm{SMod}(\Sigma^b_g)$ under $\xi$ is denoted by $\mathcal{SI}(\Sigma^b_g)$. The latter group is well known as the hyperelliptic Torelli group. Furthermore, $\mathcal{SI}(\Sigma^b_g)$ is generated by Dehn twists about symmetric separating simple closed curves that bound a subsurface of genus 1 or 2 \cite[Theorem A]{BMP}.

\subsection{Congruence subgroups of braid groups}

Let $m$ be a positive integer. The surjective homomorphisms $\mathrm{H}_1(\Sigma^1_g;\mathbb{Z}) \to \mathrm{H}_1(\Sigma^1_g;\mathbb{Z}/m)$ and $\mathrm{H}^P_1(\Sigma^2_g;\mathbb{Z}) \to \mathrm{H}^P_1(\Sigma^2_g;\mathbb{Z}/m)$ induce the following epimorphisms:
\[
\begin{cases} 
\hfill \mathrm{Sp}_{2g}(\mathbb{Z}) \to & \mathrm{Sp}_{2g}(\mathbb{Z}/m) \\
\hfill (\mathrm{Sp}_{2g+2}(\mathbb{Z}))_{y_{g+1}} \to & (\mathrm{Sp}_{2g+2}(\mathbb{Z}/m))_{y_{g+1}}.
\end{cases}
\]
Thus we have a family of representations for the braid groups
\[
\rho_m : B_{2g+b} \rightarrow
\begin{cases} 
\hfill \mathrm{Sp}_{2g}(\mathbb{Z}/m)    \hfill & \text{ if $b=1$} \\
\hfill (\mathrm{Sp}_{2g+2}(\mathbb{Z}/m))_{y_{g+1}} \hfill & \text{ if $b=2$,}
\end{cases}
\]
where $g\geq1$. The kernels of the representations $\rho_m$ are denoted by $B_{2g+b}[m]$ and they are known as \emph{level-m congruence subgroups of braid groups}.

\section{Congruence subgroups of Symplectic groups}

In this section we examine the structure of the congruence subgroups of symplectic groups.

\paragraph{Congruence subgroups and generators.} The projection $\mathbb{Z}\to \mathbb{Z}/m$ induces a surjective homomorphism $\mathrm{Sp}_{2n}(\mathbb{Z}) \rightarrow \mathrm{Sp}_{2n}(\mathbb{Z}/m)$, whose kernel is the \textit{principal level $m$ congruence subgroup} of $\mathrm{Sp}_{2n}(\mathbb{Z})$ denoted by $\mathrm{Sp}_{2n}(\mathbb{Z})[m]$. The group $\mathrm{Sp}_{2n}(\mathbb{Z})[m]$ consists of all matrices of the form $I_{2n} + m A$; where $A \in \mathrm{Sp}_{2n}(\mathbb{Z})$. Furthermore, if $m$ is a multiple of $l$ then $\mathrm{Sp}_{2n}(\mathbb{Z})[m] \triangleleft \mathrm{Sp}_{2n}(\mathbb{Z})[l]$.\\

Next we give generators for $\mathrm{Sp}_{2n}(\mathbb{Z})[p]$ when $p$ is any prime number. Let $r \in \mathbb{Z}$. We define $e_{i,j}(r)$ to be the $n \times n$ matrix with $(i,j)^{th}$ entry equal to $r$ and 0 otherwise. Let $\beta_i(r)$ be the $n \times n$ matrix with $(i,i)^{th}$ and $(i,i+1)^{th}$ entries equal to $r$, $(i+1,i+1)^{th}$ and $(i+1,i)^{th}$ entries equal to $-r$ and 0 otherwise. Define also $s e_{i,j}(r)$ to be the $n \times n$ matrix with $(i,j)^{th}$ and $(j,i)^{th}$ entries equal to $r$ and 0 otherwise. For $1 \leq i \leq j \leq n$ we define:

\[ \mathcal{X}_{i,j}(r) = I_{2n} + \left( \begin{array}{ccc}
0 & 0\\
s e_{i,j}(r) & 0\\
\end{array} \right), \quad \mathcal{Y}_{i,j}(r) = I_{2n} + \left( \begin{array}{ccc}
0 & s e_{i,j}(r)\\
0 & 0\\
\end{array} \right). \]

\begin{flushleft}
	For $1 \leq i,j \leq n$ with $i \neq j$ we define:
\end{flushleft}

\[ \mathcal{Z}_{i,j}(r) = I_{2n} + \left( \begin{array}{ccc}
e_{i,j}(r) & 0\\
0 & -e_{i,j}(r)\\
\end{array} \right). \]

\begin{flushleft}
	For $1 \leq i < n$
\end{flushleft}

\[ \mathcal{W}_{i}(r) = I_{2n} + \left( \begin{array}{ccc}
\beta_i(r) & 0\\
0 & -\beta_i(r)\\
\end{array} \right). \]

\begin{flushleft}
	Finally, 
\end{flushleft}

\[ \mathcal{U}_{1}(r) = I_{2n} + \left( \begin{array}{ccc}
e_{1,1}(r) & e_{1,1}(r)\\
-e_{1,1}(r) & -e_{1,1}(r)\\
\end{array} \right). \]

The following theorem gives a nice description of $\mathrm{Sp}_{2n}(\mathbb{Z})[p]$ as a group generated by the matrices above \cite[Lemma 5.4]{CP}.

\begin{theorem}[Church-Putman]
	For $n \geq 2$ and for a prime number $p \geq 2$ the congruence subgroup $\mathrm{Sp}_{2n}(\mathbb{Z})[p]$ is generated by the set
	\[ \mathcal{S} =  \{ \mathcal{X}_{i,j}(p), \mathcal{Y}_{i,j}(p), \mathcal{Z}_{i,j}(p), \mathcal{W}_{i}(p), \mathcal{U}_{1}(p) \} \]
	where $i,j$ are indices defined as above.
	\label{CP}
\end{theorem}

We use Theorem \ref{CP} to prove the lemma below, since we do not know a concise proof in the literature. In particular, we use the generators of Theorem \ref{CP} to prove that $\mathrm{Sp}_{2n}(\mathbb{Z}/b)$ can be expressed as a quotient of some congruence subgroup of $\mathrm{Sp}_{2n}(\mathbb{Z})$ when $b$ is a prime number. 

\begin{lemma}
	Let $a$ and $b$ two distinct prime numbers. Then the following sequence is exact.
	\[ 1 \rightarrow \mathrm{Sp}_{2n}(\mathbb{Z})[ab] \rightarrow \mathrm{Sp}_{2n}(\mathbb{Z})[a] \rightarrow \mathrm{Sp}_{2n}(\mathbb{Z}/b) \rightarrow 1. \]
	\label{Ha}
\end{lemma}

\begin{proof}
	The map $\mathrm{Sp}_{2n}(\mathbb{Z})[a] \rightarrow \mathrm{Sp}_{2n}(\mathbb{Z}/b)$ sends every matrix $A \in \mathrm{Sp}_{2n}(\mathbb{Z})[a]$ into its $\mathrm{mod}(b)$ reduction. First, we prove the surjectivity of the latter map. The generators of $\mathrm{Sp}_{2n}(\mathbb{Z}/b)$ are $\mathcal{X}_{i,j}(1) \: \mathrm{mod}(b)$ and $\mathcal{Y}_{i,j}(1) \: \mathrm{mod}(b)$ where $1 \leq i < j \leq n$. Define $n$ to be the solution of the equation $an \equiv 1 \: \mathrm{mod}(b)$. Then, $\mathcal{X}_{i,j}(a)^n \equiv \mathcal{X}_{i,j}(1) \: \mathrm{mod}(b)$ and $\mathcal{Y}_{i,j}(a)^n \equiv \mathcal{Y}_{i,j}(1) \: \mathrm{mod}(b)$. This proves the surjectivity of the reduction map. The kernel of this reduction map contains matrices which satisfy $I_{2n} + a A \equiv I_{2n} \: \mathrm{mod}(b)$. But since $a$ and $b$ are relatively primes, the latter equivalence holds if and only if $A = bB$ when $B$ is a symplectic matrix.
\end{proof}

The following proposition gives a useful decomposition of $\mathrm{Sp}_{2n}(\mathbb{Z}/m)$ \cite[Theorem 5]{NS}.

\begin{proposition}[Newman-Smart]
	\label{newman}
	Let $m \in \mathbb{N}$ and write $m = p^{k_1}_{1} p^{k_2}_{2}...p^{k_l}_{l}$, where $p^{k_i}_{i}$ are powers of prime numbers. Then
	\[ \mathrm{Sp}_{2n}(\mathbb{Z}/m) = \bigoplus^{l}_{i=1} \mathrm{Sp}_{2n}(\mathbb{Z}/p^{k_i}_{i}). \]
\end{proposition}

Newman-Smart also proved that the abelian group $\mathfrak{sp}_{2n}(\mathbb{Z}/l)$ can be expressed as a quotient of congruence subgroups of $\mathrm{Sp}_{2n}(\mathbb{Z})$, \cite[Theorem 7]{NS}.

\begin{proposition}[Newman-Smart]
	Let $l,m \geq 2$ such that $l$ divides $m$. Then we have the following isomorphism.
	$$\mathrm{Sp}_{2n}(\mathbb{Z})[m] / \mathrm{Sp}_{2n}(\mathbb{Z})[ml] \cong \mathfrak{sp}_{2n}(\mathbb{Z}/l).$$
	\label{comutsympl}
\end{proposition}

Lemma 3.2 and Propositions 3.3 and 3.4 play crucial role in Section 5, in which we explore the structure of congruence subgroups of braid groups.

\section{Topological interpretation of prime level congruence subgroups}

The purpose of this section is the characterization of the group $B_{2g+b}[p]$ when $p$ is prime. Since $B_{2g+b}\cong\mathrm{SMod}(\Sigma^b_g)$, it is convenient to study the kernel of the map
\[
 \mathrm{SMod}(\Sigma^b_{g}) \rightarrow
  \begin{cases} 
      \hfill \mathrm{Sp}_{2g}(\mathbb{Z}/p)    \hfill & \text{ if $b=1$}, \\
      \hfill (\mathrm{Sp}_{2g+2}(\mathbb{Z}/p))_{y_{g+1}} \hfill & \text{ if $b=2$}\\
  \end{cases}
\]
and we denote the map again by $\rho_p$. Also, we denote the kernel of $\rho_p$ by $B_{2g+b}[p]$.

A'Campo proved that the homomorphism $\rho_p$ is surjective \cite[Theorem 1 (1)]{C1}. Later Assion gave a presentation for $\mathrm{Sp}_{2g}(\mathbb{Z}/3)$ and $(\mathrm{Sp}_{2g+2}(\mathbb{Z}/3))_{y_{g+1}}$ as quotients of braid groups \cite{A3}. Wajnryb improved the result of Assion and generalized it for any prime number greater than 2 \cite[Theorem 1]{W}. We begin with the theorem of Wajnryb.

\begin{theorem}[Wajnryb]
\label{WA}
Consider the curves $c_i$ depicted in Figure \ref{hyperm}. Let $G_{2g+b}$ be a group with generators $T_{c_1},...,T_{c_{2g+b-1}}$ and relations $R1$ to $R6$ as follows.
\begin{enumerate}
\item[R1.] $\begin{aligned}[t]
T_{c_i} T_{c_{i+1}} T_{c_i} = T_{c_{i+1}} T_{c_i} T_{c_{i+1}};
\end{aligned}$

\item[R2.] $\begin{aligned}[t]
[T_{c_i},T_{c_j}]=1, \quad \mathrm{for} \: \vert i - j \vert >1;
\end{aligned}$

\item[R3.] $\begin{aligned}[t]
T^p_{c_1} = 1;
\end{aligned}$

\item[R4.] $\begin{aligned}[t]
(T_{c_1} T_{c_2})^6 = 1, \quad \mathrm{for} \: p > 3;
\end{aligned}$

\item[R5.] $\begin{aligned}[t]
T^{(p-1)/2}_{c_1} T^4_{c_2} T^{-(p-1)/2}_{c_1} = T^{2}_{c_2} T_{c_1} T^{-2}_{c_2},  \quad \mathrm{for} \: p > 3; \: \mathrm{and}
\end{aligned}$

\item[R6.] $\begin{aligned}[t]
(T_{c_1} T_{c_2} T_{c_3})^4 = A T^2_{c_1} A^{-1}, \: \mathrm{for} \: n > 4, \: \mathrm{where} \: A = T_{c_4} T^2_{c_3} T_{c_4} T^{(p-1)/2}_{c_2} T^{-1}_{c_3} T_{c_2}. 
\end{aligned}$
\end{enumerate}
Then $G_{2g+1}$ is isomorphic to $\mathrm{Sp}_{2g}(\mathbb{Z}/p)$, and $G_{2g+2}$ is isomorphic to $(\mathrm{Sp}_{2g+2}(\mathbb{Z}/p))_{y_{n+1}}$.
\end{theorem}

As a consequence of Theorem \ref{WA} we obtain elements of $\mathrm{SMod}(\Sigma^b_{g})$ which normally generate $B_{2g+b}[p]$.

%

In the rest of the section we examine the elements of the relations of Theorem \ref{WA} in order to give a topological description for the generators of $B_n[p]$. We note that relations $R1$ and $R2$ are the defining relations in the presentation of the braid group.

We denote by $[c_i]$ the homology class of $c_i$, and by $T_{[c_i]}$ the transvection associated to the Dehn twist $T_{c_i}$ under the map
\[
  \mathrm{SMod}(\Sigma^b_{g}) \rightarrow
  \begin{cases} 
      \hfill \mathrm{Sp}_{2g}(\mathbb{Z}/p)    \hfill & \text{ if $b=1$}, \\
      \hfill (\mathrm{Sp}_{2g+2}(\mathbb{Z}/p))_{y_{g+1}} \hfill & \text{ if $b=2$}. \\
  \end{cases}
\]
By definition, the action of a transvection $T^m_{[c]}$ on an element $u \in \mathrm{H}_1(\Sigma^1_{g},\mathbb{Z})$ (respectively $\mathrm{H}^P_1(\Sigma^2_{g},\mathbb{Z})$) is defined to be $T^m_{[c]}(u) = [u] + m \hat{i}(u,[c])[c]$, where $\hat{i}$ stands for the algebraic intersection number.

\paragraph{\textbf{R3}: Powers of Dehn twists.} The $p^{th}$ powers of Dehn twists about symmetric nonseparating simple closed curves are easy to check by looking at their image in the symplectic group. The symplectic representation sends $T^p_{c_1}$ into the following matrix:

\[ \left( \begin{array}{ccc}
1 & p\\
0 & 1\\
\end{array} \right) \oplus I,  \]
where $I$ stands for the identity matrix of dimension depending on $g$ and $b$ (see Section 7.1.3). The $\mathrm{mod}(p)$ reduction of the matrix above is the identity. Moreover, every Dehn twist about a non-separating curve is conjugate to $T_{c_1}$. As a consequence, every Dehn twist in $\mathrm{SMod}(\Sigma^b_{g})$ raised to the power of $p$ lies in $B_n[p]$.

\paragraph{$\textbf{R4}$: Symmetric separating Dehn twists.} By the chain relation the element $(T_{c_1} T_{c_2})^6$ can be represented by a Dehn twist $T_{\gamma}$, where $\gamma$ is the symmetric separating curve bounding the genus 1 subsurface of $\Sigma^b_{g}$ as indicated in Figure \ref{refive} \cite[Proposition 4.12]{BFM}. We can generalize the relation R4 by considering a symmetric separating curve $\delta$ of a genus $k$ subsurface of $\Sigma^b_{g}$. By the chain relation there is a maximal chain of curves $a_1,...,a_{2k}$ in the subsurface of genus $k$ with boundary $\delta$ such that $(T_{a_1}...T_{a_{2k}})^{4k+2} = T_{\delta}$.

The fact that every symmetric separating simple closed curve $\delta$ is nullhomologous in $H_1(\Sigma^1_g)$ (respectively $H^P_1(\Sigma^2_g)$) implies that $T_{[\delta]}(x)=x + \iota_a(x,[\delta])=x+0=x$ for every $x \in H_1(\Sigma^1_g)$ (respectively $H^P_1(\Sigma^2_g)$), where $T_{[\delta]}$ is the corresponding transvection of $T_{\delta}$ as described in Section 2. Since for every symmetric separating curve $\delta$ in $\Sigma^b_{g}$ and $T_{\delta} \in B_{2g+b}[p]$ we have that $(T_{a_1}...T_{a_{2k}})^{4k+2} \in \mathcal{SI}(\Sigma^b_{g}) \subset  B_{2g+b}[p]$.

\begin{figure}[h]
\begin{center}
\includegraphics[scale=.3]{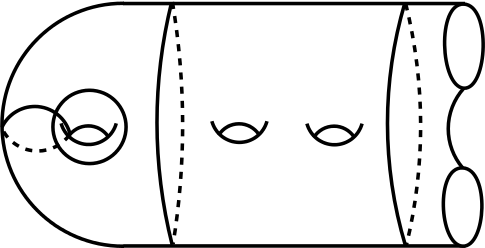}
\end{center}
\caption{The curve $\gamma$ that bound a surface of genus 1.}
\label{refive}
\begin{picture}(22,12)
\put(190,40){$\gamma$}
\put(143,75){$c_1$}
\put(170,62){$c_2$}
\end{picture}
\end{figure}
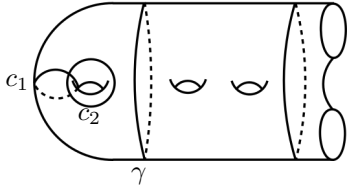

\paragraph{$\textbf{R5}$: Mod-p involution maps.} We begin by modifying the relation $R5$ of Theorem \ref{WA}.

\begin{lemma}
The relation $R5$ given above is equivalent to:
$$(T^{(p+1)/2}_{c_1} T^{4}_{c_2})^2 = (T_{c_1} T_{c_2})^3$$
in $\mathrm{Sp}_{2g}(\mathbb{Z}/p)$ (respectively $(\mathrm{Sp}_{2g+2}(\mathbb{Z}/p))_{y_{g+1}}$).
\label{insert}
\end{lemma}

\begin{proof}
We have that $(T_{c_1} T_{c_2})^3 = T_{c_1} T^{2}_{c_2} T_{c_1} T^{2}_{c_2}$. Then
\[ T^{(p-1)/2}_{c_1} T^4_{c_2} T^{-(p-1)/2}_{c_1} = T^{-1}_{c_1} (T^{(p+1)/2}_{c_1} T^4_{c_2})^2 T^{-4}_{c_2} = T^{2}_{c_2} T_{c_1} T^{-2}_{c_2}. \]
On the other hand
\[ (T^{(p+1)/2}_{c_1} T^{4}_{c_2})^2 = T_{c_1} T^{(p-1)/2}_{c_1} T^4_{c_2} T^{-(p-1)/2}_{c_1} T^4_{c_2} = T_{c_1} T^2_{c_2} T_{c_1} T^2_{c_2}. \] 
\end{proof}

Now we examine the relation of Lemma \ref{insert}. 
\paragraph{RHS.} For $i=1,2$, $(T_{c_1} T_{c_2})^3([c_i]) = -[c_i]$, where $[c_i]$ stands for the homology class of $c_i$. Thus, the homeomorphism $(T_{c_1} T_{c_2})^3$ acts as the hyperelliptic involution on the subsurface bounded by the boundary of the chain $ch(c_1,c_2)$ (see Figure \ref{refive}).
\paragraph{LHS.} We have
\begin{align*}
 (T^{(p+1)/2}_{c_1} T^{4}_{c_2})^2([c_1]) = -8p[c_2] + (4p^2 +2p -1)[c_1] \equiv -[c_1] \: \bmod(p),\\
(T^{(p+1)/2}_{c_1} T^{4}_{c_2})^2([c_2]) = 2p \frac{p+1}{2} [c_1] - (2p+1)[c_2] \equiv -[c_2] \: \bmod(p)
\end{align*}
Therefore, $(T^{(p+1)/2}_{c_1} T^{4}_{c_2})^2$ acts as the hyperelliptic involution $\bmod(p)$ in the subspace of $\mathrm{H}_1(\Sigma^1_g,\mathbb{Z}/p)$ (resp $\mathrm{H}^P_1(\Sigma^2_g,\mathbb{Z}/p)$) spanned by $[c_1],[c_2]$.\\

We can generalize Relation $R5$ as follows. For $k$ even, consider any chain $ch(a_1,a_2,...,a_k)$ of symmetric simple closed curves such that $T_{a_i} \in \mathrm{SMod}(\Sigma_{g,b})$ for all $i \leq k$. Choose an $f \in \mathrm{SMod}(\Sigma^b_g)$ such that $f([a_i]) = -[a_i]$. Then $(T_{a_1}...T_{a_k})^{k+1} f^{-1} \in B_{2g+b}[p]$. We call this type of element an \emph{mod-p involution map}.

\paragraph{$\textbf{R6}$: Mod-p center maps.} We describe a generalized version of $(T_{c_1} T_{c_2} T_{c_3})^4 (A T^{-2}_{c_1} A^{-1})$. Let $A_1$ be the trivial homeomorphism in $\mathrm{SMod}(\Sigma^b_{g})$. For $k$ odd, and $k\geq3$, define
\[ A_k = T_{c_{k+1}} T^{2}_{c_k} T_{c_{k+1}} T^{(p-1)/2}_{c_{k-1}} T^{-1}_{c_k} T_{c_{k-1}} A_{k-2}. \]
First, we deal with the case $b=1$. (For $b=2$ the process is exactly the same.) Consider the symplectic bases $\{ y_i,x_i \}$ for $\mathrm{H}_1(\Sigma^1_{g},\mathbb{Z})$ depicted on Figure \ref{sympb1}.

\begin{lemma}
For $k$ odd, we have that $A_k T_{[c_1]} A^{-1}_k = T_{[y_{(k+1)/2}]}$ in $\mathrm{Sp}_{2g}(\mathbb{Z}/p)$.
\label{rel6}
\end{lemma}

Note that if $k=3$, then $T_{[y_2]} = T_{[d_3]}$.

\begin{proof}
We need to prove that $A_k([c_1]) \equiv [c_1] + [c_3] + ... + [c_k] \in \mathrm{Sp}_{2g}(\mathbb{Z}/p)$. A direct calculation shows that $A_3 ([c_1]) \equiv [c_1] + [c_3] \: \mathrm{mod}(p)$. Assume that the theorem is true for $k-2$, that is $A_{k-2}([c_1]) = [c_1] + [c_3] + ... + [c_{k-2}]$. Then $T_{c_{k+1}} T^{2}_{c_k} T_{c_{k+1}} T^{(p-1)/2}_{c_{k-1}} T^{-1}_{c_k} T_{c_{k-1}}([c_{k-2}]) \equiv [c_{k-2}] + [c_k] \bmod(p)$. The proof of the lemma follows.
\end{proof}

\begin{figure}[h]
\begin{center}
\includegraphics[scale=.35]{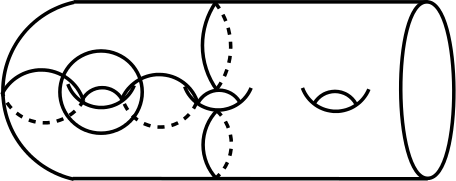}
\end{center}
\caption{The chain relation of $R6$.}
\begin{picture}(22,12)
\put(137,72){$c_1$}
\put(170,88){$c_2$}
\put(188,81){$c_3$}
\put(215,85){$d_3$}
\put(215,57){$d'_3$}
\end{picture}
\label{skatochain}
\end{figure}
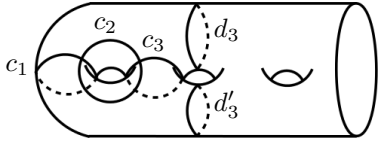

Let $k$ be an odd integer, and consider also the odd chain $ch(c_1,c_2,...,c_k)$. By the chain relation we have that $(T_{c_1}...T_{c_k})^{k+1} = T_{d_k} T_{d'_k}$, where $d_k = y_{(k+1)/2}$, and $[d_k]=[d'_k]=[y_{(k+1)/2}]$ (see, for example, Figure \ref{skatochain}). Thus, $(T_{[c_1]}...T_{[c_k]})^{k+1} = T^2_{[y_{(k+1)/2}]} \in \mathrm{Sp}_{2g}(\mathbb{Z}/p)$. On the other hand, according to Lemma \ref{rel6} we have that $A_k T^2_{[c_1]} A^{-1}_k = T^2_{[y_{(k+1)/2}]} \in \mathrm{Sp}_{2g}(\mathbb{Z}/p)$. Hence, $(T_{c_1}...T_{c_k})^{k+1} A_k T^{-2}_{c_1} A^{-1}_k \in B_{n}[p]$. Note that if $k=3$, the element $(T_{c_1}...T_{c_k})^{k+1} A_k T^{-2}_{c_1} A^{-1}_k$ is the same one as in the relation 6 of Theorem \ref{WA}.

We can describe a generalized version of $(T_{c_1}...T_{c_k})^{k+1} A_k T^{-2}_{c_1} A^{-1}_k$. Consider any odd chain $ch(a_1,a_2,...,a_k)$, such that $T_{a_i} \in \mathrm{SMod}(\Sigma^1_{g})$ for all $i \leq k$. Choose a homeomorphism $h \in \mathrm{SMod}(\Sigma^1_{g})$ such that $h([a_1]) = [a_1]+[a_3]+...+[a_k] \in \mathrm{Sp}_{2g}(\Sigma^1_{g})$. Then $(T_{a_1}...T_{a_k})^{k+1} h T^{-2}_{a_1} h^{-1}$ lies on $B_{2g+1}[p]$. If we consider $(T_{a_1}...T_{a_k})^{k+1}$ as the center of the subgroup $K$ of $\mathrm{SMod}(\Sigma^b_{g})$ generated by $T_{a_1}...T_{a_k}$, then $h T^{-2}_{a_1} h^{-1}$ is the center $\bmod(p)$ of the same group. Note that the choice of $h$ is not unique. We call this type of element an \emph{mod-p center map}.

\paragraph{Generators for congruence subgroups.} As a corollary of Theorem \ref{WA} we obtain the following theorem.

\begin{theorem}
If $p=3$, then $B_{2g+b}[3]$ is generated by Dehn twists raised to the power of $3$, and for $2g+b>4$ by mod-p center maps. For $p>3$ the subgroup $B_{2g+b}[p]$ of $\mathrm{SMod}(\Sigma^b_{g})$ is generated by Dehn twists raised to the power of $p$, by Dehn twists about symmetric separating curves, by mod-p involution maps, and for $2g+b>4$ by mod-p center maps.
\label{congen}
\end{theorem}

\paragraph{Finite set of generators.} It is well known that every finite index subgroup of a finitely generated group, is finitely generated \cite[Corollary 2.7.1]{MKS}. The generating set in Theorem \ref{congen} is infinite. When $p=3$ and $g=1$ we can find a finite set of generators.

\begin{theorem}
The group $B_3[3]$ is generated by four elements.
\label{cong3}
\end{theorem}

\begin{proof}
Set $S = \{ T^3_{c_1},T^3_{c_2}, T_{c_2} T^3_{c_1} T^{-1}_{c_2}, T^2_{c_2} T^3_{c_1} T^{-2}_{c_2} \}.$ We denote by $\Gamma$ the subgroup of $B_3[3]$ generated by $S$. We prove that if we conjugate elements of $S$ by $T_{c_1}$ or $T_{c_2}$, then the resulting elements lie in $\Gamma$. Since $B_3[3]$ is normally generated by $S$ and since $S$ generates a normal subgroup of $B_3$, then $\Gamma = B_3[3]$.\\

In the braid group we have the relation $$T_{c_j} T_{c_{j-1}}...T^3_{c_i}...T^{-1}_{c_{j-1}} T^{-1}_{c_j} =  T^{-1}_{c_i} T^{-1}_{c_{i+1}}...T^3_{c_j}...T_{c_{i+1}} T_{c_i} $$
We prove the theorem in three steps.

\paragraph{Step 1:} Conjugates of $T^3_{c_1},T^3_{c_2}$:\\
\begin{align*}
T^{- 1}_{c_2} T^3_{c_1} T_{c_2} = T^{-3}_{c_2} T^{2}_{c_2} T^3_{c_1} T^{-2}_{c_2} T^{3}_{c_2} \in \Gamma\\
T^{-1}_{c_1} T^3_2 T_{c_1} = T_2 T^3_{c_1} T^{-1}_2 \in \Gamma\\
 T_{c_1} T^3_{c_2} T^{-1}_{c_1} =  T^{- 1}_{c_2} T^3_{c_1} T_{c_2} = T^{-3}_{c_2} T^{2}_{c_2} T^3_{c_1} T^{-2}_{c_2} T^{3}_{c_2} \in \Gamma.
\end{align*}
\paragraph{Step 2:} Conjugates of $T_{c_2} T^3_{c_1} T^{-1}_{c_2}$:
\begin{align*}
T_{c_1} T_{c_2} T^3_{c_1} T^{-1}_{c_2} T^{-1}_{c_1} = T^3_{c_2} \in \Gamma\\
T^{-1}_{c_1} T_{c_2} T^3_{c_1} T^{-1}_{c_2} T_{c_1} = T^{-2}_{c_1} T^{3}_{c_2} T^{2}_{c_1} = T^{-3}_{c_1} (T_{c_1} T^3_{c_2} T^{-1}_{c_1}) T^3_{c_1}.
\end{align*}
The latter is in $\Gamma$ by step 1.

\paragraph{Step 3:} Conjugates of $T^2_{c_2} T^3_{c_1} T^{-2}_{c_2}$:
\begin{align*}
 T^{-1}_{c_1} T^2_{c_2} T^3_{c_1} T^{-2}_{c_2} T_{c_1} = T^{-1}_{c_1} T^3_{c_2} T^{-1}_{c_2} T^3_{c_1} T_{c_2} T^{-3}_{c_2} T_{c_1} = \\
(T^{-1}_{c_1} T^3_{c_2} T_{c_1}) (T^{-1}_{c_1} T^{-1}_{c_2} T^3_{c_1} T_{c_2} T_{c_1}) (T^{-1}_1 T^{-3}_{c_2} T_{c_1}) 
\end{align*}
The elements $(T^{-1}_{c_1} T^3_{c_2} T_{c_1}), (T^{-1}_{c_1} T^{-3}_{c_2} T_{c_1})$ are in $\Gamma$ by step 1.
\begin{align*}
T^{-1}_{c_1} T^{-1}_{c_2} T^3_{c_1} T_{c_2} T_{c_1} = T^3_{c_2}
\end{align*}
Finally, since $T^2_{c_2} T^3_{c_1} T^{-2}_{c_2} = T^3_{c_2} T^{-1}_{c_2} T^3_{c_1} T_{c_2} T^{-3}_{c_2}$, it suffices to check that $T_{c_1} T^{-1}_{c_2} T^3_{c_1} T_{c_2} T^{-1}_{c_1}$ is in $\Gamma$. But we have that
\begin{align*}
T_{c_1} T^{-1}_{c_2} T^3_{c_1} T_{c_2} T^{-1}_{c_1} = T^2_{c_1} T^3_{c_2} T^{-2}_{c_1} = T^3_{c_1} T^{-1}_{c_1} T^3_2 T_{c_1} T^{-3}_{c_1}=T^3_{c_1} T_{c_2} T^3_{c_1} T^{-1}_{c_2} T^{-3}_{c_1} \in \Gamma.
\end{align*}
This proves the theorem.
\end{proof}

Since $T^2_{c_2} T^3_{c_1} T^{-2}_{c_2} = T^3_{c_2} T^{-1}_{c_2} T^3_{c_1} T_{c_2} T^{-3}_{c_2}$ we deduce that $\{ T^3_{c_1},T^3_{c_2}, T_{c_2} T^3_{c_1} T^{-1}_{c_2}, T^{-1}_{c_2} T^3_{c_1} T_{c_2} \}$ is also a generating set for $B_3[3]$.

\section{Symplectic groups and pure braid groups}

For $i \in \mathbb{N}$, let $p_i$ denote a prime number greater than 2. In this section we characterize $B_{2g+b}[m]$, where $m=2p_1 p_2...p_k$ and $m=4p_1 p_2...p_k$. Our strategy is to find a presentation for $PB_{2g+b}/B_{2g+b}[m]$. We recall that $\mathrm{H}_1(PB_{2g+b},\mathbb{Z}/2)$ is $\mathfrak{sp}_{2g}(\mathbb{Z}/2)$, if $b=1$ and $\mathrm{Ann}(y_{g+1})$ if $b=2$, where $\mathrm{Ann}(y_{g+1}) = \{ h \in \mathfrak{sp}_{2g+2}(\mathbb{Z}/2) \mid h(y_{g+1})=0 \}$ \cite{BM}. The generators of $B_{2g+b}$ are denoted by $\sigma_i$ and the generators of $PB_{2g+b}$ are denoted by $a_{i,j}$ as in Section 2.

\begin{theorem}
	For $m=2p_1 p_2...p_k$, where $p_i\geq3$ are prime numbers, we have
	
	\[
	PB_{2g+b}/B_{2g+b}[m] = 
	\begin{cases} 
	\hfill  \bigoplus^k_{i=1} \mathrm{Sp}_{2g}(\mathbb{Z}/p_i)    \hfill & \text{ if $b=1$}, \\
	\hfill  \bigoplus^k_{i=1} (\mathrm{Sp}_{2g+2}(\mathbb{Z}/p_i))_{y_{g+1}} \hfill & \text{ if $b=2$}.\\
	\end{cases}
	\]
	\label{PBQ1}
\end{theorem}

\begin{proof}
	
	We set $m=2p_1 p_2 ... p_k$. We have the map
	\[
	\rho_m: B_{2g+b} \rightarrow 
	\begin{cases} 
	\hfill \mathrm{Sp}_{2g}(\mathbb{Z}) \rightarrow \mathrm{Sp}_{2g}(\mathbb{Z}/m)    \hfill & \text{ if $b=1$}, \\
	\hfill (\mathrm{Sp}_{2g+2}(\mathbb{Z}))_{y_{g+1}} \rightarrow (\mathrm{Sp}_{2g+2}(\mathbb{Z}/m))_{y_{g+1}} \hfill & \text{ if $b=2$}\\
	\end{cases}
	\]
	with kernel $B_{2g+b}[m]$. By Lemma \ref{newman} we know that 
	\[ \mathrm{Sp}_{2g}(\mathbb{Z}/m)  = \mathrm{Sp}_{2g}(\mathbb{Z}/2) \bigoplus^{k}_{i=1} \mathrm{Sp}_{2g}(\mathbb{Z}/p_i). \]
	If we restrict to the pure braid group, then the image of the map $PB_{2g+1} \rightarrow \mathrm{Sp}_{2g}(\mathbb{Z})$ is the group $\mathrm{Sp}_{2g}(\mathbb{Z})[2]$, (see \cite[Theorem 3.3]{BM}). Furthermore, by Lemma \ref{Ha} we have that the map $\mathrm{Sp}_{2g}(\mathbb{Z})[2] \rightarrow \mathrm{Sp}(\mathbb{Z}/p_i)$ is surjective. Thus, the image of the map
	$$\mathrm{Sp}_{2g}(\mathbb{Z}) \rightarrow \mathrm{Sp}_{2g}(\mathbb{Z}/m) = \mathrm{Sp}_{2g}(\mathbb{Z}/2) \bigoplus^{k}_{i=1} \mathrm{Sp}_{2g}(\mathbb{Z}/p_i),$$
	after we restrict to $\mathrm{Sp}_{2g}(\mathbb{Z})[2]$, is the group $\bigoplus^{k}_{i=1} \mathrm{Sp}_{2g}(\mathbb{Z}/p_i)$. Hence, have a short exact sequence
	\[ 1 \rightarrow B_{2g+1}[m] \rightarrow PB_{2g+1} \rightarrow  \bigoplus^{k}_{i=1} \mathrm{Sp}_{2g}(\mathbb{Z}/p_i) \rightarrow 1.\]
	Likewise, since the image of the map $PB_{2g+2} \rightarrow (\mathrm{Sp}_{2g+2}(\mathbb{Z}))_{y_{g+1}}$ is $(\mathrm{Sp}_{2g+2}(\mathbb{Z})[2])_{y_{g+1}}$ (see \cite[Theorem 3.3]{BM}), and since $(\mathrm{Sp}_{2g+2}(\mathbb{Z}/m))_{y_{g+1}} <\mathrm{Sp}_{2g+2}(\mathbb{Z}/m)$, we can apply Lemma \ref{newman} and end up with the following exact sequence.
	\[ 1 \rightarrow B_{2g+2}[m] \rightarrow PB_{2g+2} \rightarrow  \bigoplus^k_{i=1} (\mathrm{Sp}_{2g+2}(\mathbb{Z}/p_i))_{y_{g+1}} \rightarrow 1.\]\\
	This completes the proof.
\end{proof}

In the following statement we slightly generalize Lemma \ref{PBQ1}. The symplectic Lie algebra $\mathfrak{sp}_{2n}(\mathbb{Z})$ consists of those elements $A \in \mathfrak{gl}_{2n}(\mathbb{Z})$ which satisfy the relation $A^T J + J A = 0$. We define also
\[ \mathrm{Ann}(u) = \{ m \in \mathfrak{sp}_{2n}(\mathbb{Z})\mid m(u)=0 \}, \]
where $\mathrm{Ann}(u)$ stands for the annihilator of the vector $u$. We have the following theorem.

\begin{theorem}
	For $m=4p_1 p_2...p_k$, where $p_i\geq3$ are prime numbers , we have
	
	\[
	PB_{2g+b}/B_{2g+b}[m] = 
	\begin{cases} 
	\hfill \mathfrak{sp}_{2g}(\mathbb{Z}/2) \bigoplus^k_{i=1} \mathrm{Sp}_{2g}(\mathbb{Z}/p_i)    \hfill & \text{ if $b=1$}, \\
	\hfill \mathrm{Ann}(e) \bigoplus^k_{i=1} (\mathrm{Sp}_{2g+2}(\mathbb{Z}/p_i))_{y_{g+1}} \hfill & \text{ if $b=2$}.\\
	\end{cases}
	\]
	\label{PBQ2}
\end{theorem}

\begin{proof}
	We set $m=4p_1 p_2 ... p_k$. By Lemma \ref{newman} we have that 
	\[ \mathrm{Sp}_{2g}(\mathbb{Z}/m)  = \mathrm{Sp}_{2g}(\mathbb{Z}/4) \bigoplus^{k}_{i=1} \mathrm{Sp}_{2g}(\mathbb{Z}/p_i). \]
	We want to characterize the image of the map
	\[
	B_{2g+b} \rightarrow 
	\begin{cases} 
	\hfill \mathrm{Sp}_{2g}(\mathbb{Z}/4) \bigoplus^k_{i=1} \mathrm{Sp}_{2g}(\mathbb{Z}/p_i)    \hfill & \text{ if $b=1$}, \\
	\hfill (\mathrm{Sp}_{2g+2}(\mathbb{Z}/4))_{y_{g+1}} \bigoplus^k_{i=1} (\mathrm{Sp}_{2g+2}(\mathbb{Z}/p_i))_{y_{g+1}} \hfill & \text{ if $b=2$}.\\
	\end{cases}
	\]
	
	For $b=1$ we only need to characterize the image of the restriction of the map above to $PB_{2g+b}$. In particular, we want to compute the image of the map $PB_{2g+1} \rightarrow \mathrm{Sp}_{2g}(\mathbb{Z}/4)$. We know that the image of the map $PB_{2g+1} \rightarrow \mathrm{Sp}_{2g}(\mathbb{Z})$ is  $\mathrm{Sp}_{2g}(\mathbb{Z})[2]$. Consider the inclusion
	\[ \mathrm{Sp}_{2g}(\mathbb{Z})[2] \hookrightarrow \mathrm{Sp}_{2g}(\mathbb{Z}). \]
	We quotient the above inclusion by $\mathrm{Sp}_{2g}(\mathbb{Z})[4]$, and we get the following inclusion:
	\[ \mathfrak{sp}_{2g}(\mathbb{Z}/2) \hookrightarrow \mathrm{Sp}_{2g}(\mathbb{Z}/4). \]
	We finally have
	\[ PB_{2g+1} \rightarrow \mathrm{Sp}_{2g}(\mathbb{Z})[2] \rightarrow  \mathfrak{sp}_{2g}(\mathbb{Z}/2) < \mathrm{Sp}_{2g}(\mathbb{Z}/4). \]
	
	Hence, the image of the map $PB_{2g+1} \rightarrow \mathrm{Sp}_{2g}(\mathbb{Z}/4)$ is the abelian group $\mathfrak{sp}_{2g}(\mathbb{Z}/2)$. Thus, we have
	\[ PB_{2g+b}/B_{2g+b}[m] \cong \mathfrak{sp}_{2g}(\mathbb{Z}/2) \bigoplus^k_{i=1} \mathrm{Sp}_{2g}(\mathbb{Z}/p_i). \]
	
	For $b=2$, the maps
	\[ PB_{2g+2} \rightarrow (\mathrm{Sp}_{2g+2}(\mathbb{Z})[2])_{y_{g+1}} \rightarrow \mathrm{Ann}(y_{g+1}) \]
	are both surjective, \cite[Lemma 3.5]{BM}. But $\mathrm{Ann}(y_{g+1}) < (\mathrm{Sp}_{2g+2}(\mathbb{Z}/4))_{y_{g+1}}$, and thus, the image of the map 
	\[ PB_{2g+2} \rightarrow (\mathrm{Sp}_{2g+2}(\mathbb{Z}/4))_{y_{g+1}} \]
	is the group $\mathrm{Ann}(y_{g+1})$. Thus, we get
	\[ PB_{2g+2} / B_{2g+2}[m] \cong \mathrm{Ann}(y_{g+1}) \bigoplus^k_{i=1} (\mathrm{Sp}_{2g+2}(\mathbb{Z}/p_i))_{y_{g+1}}.  \]
	This completes the proof.
\end{proof}

In order to find generators for $B_{2g+1}[m]$, it suffices to find a presentation for $\mathrm{Sp}_{2g}(\mathbb{Z}/p)$ in terms of pure braids. In the next proposition we prove that $\mathrm{Sp}_{2g}(\mathbb{Z}/p)$ admits a presentation as a quotient of the pure braid group over some relations. These new relations are the generators for $B_{2g+1}[2p]$. Recall that the generators of $PB_n$ are defined to be $a_{i,j} = \sigma_{j-1}...\sigma_{i+1} \sigma^2_i \sigma^{-1}_{i+1}...\sigma^{-1}_{j-1}$, where $1 \leq i < j \leq n$.

\begin{proposition}
	Fix a prime number $p$, and put $p=2k+1$. Let $H_n$ be the group with generators $ \{ a_{i,j} \}$ with defining relations as follows:
	\label{SYPRE}
	\begin{enumerate}
		\item[PR1.] $\begin{aligned}[t]
		a^{k}_{i,i+1} a^{k}_{i+1,i+2} a^{k}_{i,i+1}= a^{k}_{i+1,i+2} a^{k}_{i,i+1} a^{k}_{i+1,i+2},
		\end{aligned}$
		
		\item[PR2.] $\begin{aligned}[t]
		a^p_{i,j} = 1,
		\end{aligned}$
		
		\item[PR3.] $\begin{aligned}[t]
		(a_{1,2} a_{1,3} a_{2,3})^2=1 \: \mathrm{for} \, p>3,
		\end{aligned}$
		
		\item[PR4.] $\begin{aligned}[t]
		a^{-1}_{r,s} a_{i,j} a_{r,s}  = a_{i,j}, \: 1 \leq r<s<i<j \leq n \: \mathrm{or} \: 1 \leq i<r<s<j \leq n,
		\end{aligned}$
		
		\item[PR5.] $\begin{aligned}[t]
		a^{-1}_{r,s} a_{i,j} a_{r,s}  = a_{r,j} a_{i,j} a^{-1}_{r,j} , \: 1 \leq r<s=i<j \leq n,
		\end{aligned}$
		
		\item[PR6.] $\begin{aligned}[t]
		a^{-1}_{r,s} a_{i,j} a_{r,s}  = (a_{i,j} a_{s,j})a_{i,j}(a_{i,j} a_{s,j})^{-1} , \: 1 \leq r=i<s<j \leq n,
		\end{aligned}$
		
		\item[PR7.] $\begin{aligned}[t]
		a^{-1}_{r,s} a_{i,j} a_{r,s}  = (a_{r,j} a_{s,j} a^{-1}_{r,j} a^{-1}_{s,j})a_{i,j}(a_{r,j} a_{s,j} a^{-1}_{r,j} a^{-1}_{s,j})^{-1} , \: 1 \leq r<i<s<j \leq n,
		\end{aligned}$
		
		\item[PR8.] $\begin{aligned}[t]
		a_{i,j} = a^{k+1}_{j-1,j} a^{k+1}_{j-2,j-1}...a_{i,i+1} a^k_{i+1,i+2}...a^k_{j-1,j}, \: 1 < |i-j| \leq n,
		\end{aligned}$
		
		\item[PR9.] $\begin{aligned}[t]
		a_{1,2} a_{1,3} a_{2,3} = C, \mathrm{where}
		\end{aligned}$
		\subitem $C = (a^{(p+1)/4}_{1,2} a^2_{2,3})^2$, if $(p+1)/2$ is even,
		\subitem $C = a^{(p+3)/4}_{1,2} a^2_{1,3} a^{(p-1)/4}_{1,2} a^2_{2,3}$, if $(p+1)/2$ is odd.
		
		\item[PR10.] $\begin{aligned}[t]
		a_{1,2} a_{1,3} a_{1,4} a_{2,3} a_{2,4} a_{3,4} = B a_{1,4} B^{-1}, \mathrm{where}
		\end{aligned}$
		\subitem $B = a_{3,5} a_{4,5} a^{k/2}_{2,3} a^{-1}_{3,4}$, if $k$ is even,
		\subitem $B = a_{3,5} a_{4,5} a^{k+1}_{2,3} a_{3,4}$, if $k$ is odd.		
	\end{enumerate}
	
	\begin{flushleft}
		If $n = 2g+1$ then $H_{n}$ is isomorphic to $\mathrm{Sp}_{2g}(\mathbb{Z}/p)$. On the other hand if $n= 2g+2$, then $H_{n}$ is isomorphic to $\mathrm{Sp}_{2g+2}(\mathbb{Z}/p)_{y_{g+1}}$.
	\end{flushleft}
	
\end{proposition}


Note that relations $PR4$, $PR5$, $PR6$, $PR7$ are relations in the presentation of the pure braid group given in Chapter 4. We begin with the group $G_n$ defined in Theorem \ref{WA}, and using Tietze transformations, we obtain the presentation of $H_n$.\\

\begin{proof}
	By Theorem \ref{WA} the group $G_n$ has the following presentation:
	\[ G_n = \langle \sigma_i \vert \: R1,R2,R3,R4,R5,R6 \rangle, \]
	where $1 \leq i < 2g+b$. Let $a_{i,j} = \sigma_{j-1}...\sigma_{i+1} \sigma^2_i \sigma^{-1}_{i+1}...\sigma^{-1}_{j-1}$ and denote this relation by $PR11$. Then include $PR11$ into the presentation of $G_n$ and add the generator $a_{i,j}$ to obtain
	\[ \langle \sigma_i, a_{i,j} \vert \: R1,R2,R3,R4,R5,R6, PR11 \rangle.\]
	Since $PB_n$ is a subgroup of $B_n$, this means that $R1$ and $R2$ can be used to deduce the relations $PR4$, $PR5$, $PR6$, $PR7$.
	\[ \langle \sigma_i, a_{i,j} \vert \: R1,R2,R3,R4,R5,R6,PR4,PR5,PR6,PR7, PR11 \rangle.\]
	The relation $R2$ can be deduced by $PR11$ and $R3$ and $PR4$
	\[ \langle \sigma_i, a_{i,j} \vert \: R1,R3,R4,R5,R6,PR2,PR4,PR5,PR6,PR7, PR11 \rangle.\]
	We derive two more relations from $PR11$ and $R3$.
	\[ \sigma_i = a^{k+1}_{i,i+1}, \quad \sigma^{-1}_{i} = a^k_{i,i+1}. \]
	Then $PR1$ is equivalent to $R1$, $PR2$ is equivalent to $R3$, $PR3$ is equivalent to $R4$, $PR9$ is equivalent to $R5$, $PR10$ is equivalent to $R6$, and $PR11$ is equivalent to $PR8$. In other words,
	\[ \langle \sigma_i, a_{i,j} \vert \: PR1,PR2,PR4,PR5,PR6,PR7,PR8,PR9,PR10,\sigma_i = a^{k+1}_{i,i+1}, \sigma^{-1}_{i} = a^k_{i,i+1} \rangle \]
	Finally, for $1 \leq i < j \geq 2g+b$ we have that
	\[ \langle a_{i,j} \vert \: PR1,PR2,PR4,PR5,PR6,PR7,PR8,PR9,PR10 \rangle, \]
	which is the presentation of $H_n$.
\end{proof}

As an application of Proposition \ref{SYPRE}, we can obtain generators for $B_{2g+b}[2p]$.

\begin{corollary}
	For $k = (p-1)/2$, the group $B_{2g+b}[2p]$ is normally generated by six types of elements:
	\begin{align*}
		a^p_{i,j},\\
		(a_{1,2} a_{1,3} a_{2,3})^2,\\
		a_{1,2} a_{1,3} a_{2,3} C^{-1},\\
		a_{1,2} a_{1,3} a_{1,4} a_{2,3} a_{2,4} a_{3,4} B a^{-1}_{1,4} B^{-1},\\
		a^{k}_{i,i+1} a^{k}_{i+1,i+2} a^{k}_{i,i+1}  a^{-k}_{i+1,i+2} a^{-k}_{i,i+1} a^{-k}_{i+1,i+2},\\
		a^{k+1}_{j-1,j} a^{k+1}_{j-2,j-1}...a_{i,i+1} a^k_{i+1,i+2}...a^k_{j-1,j} a^{-1}_{i,j}.\\
	\end{align*}
\end{corollary}

Actually we can use Proposition \ref{SYPRE} to find normal generators for any $B_n[m]$, where $m$ is either $2 p_1 ... p_k$ or $4 p_1 ... p_k$ and $p_i\geq3$ are prime numbers.

\section{Symmetric quotients of congruence subgroups}

In this section we explore factor groups of congruence subgroups of braid groups. From Section 3 we know that $B_n[2] \cong PB_n$ and $B_n / B_n[2] \cong S_n$. In the next theorem we generalize the latter isomorphism.

\begin{theorem}
The quotient $B_n[p] / B_n[2p]$ is isomorphic to $S_n$.
\label{symquo}
\end{theorem}
Before we proceed to the proof of Theorem \ref{symquo}, we will prove the following lemma.

\begin{lemma}
The groups $B_n[2p]$ and $B_n[2] \cap B_n[p]$ are isomorphic.
\label{evenint}
\end{lemma}

\begin{proof}
It is obvious that $B_n[2p] < B_n[2] \cap B_n[p]$. By Proposition \ref{newman} we have the decomposition $\mathrm{Sp}_{2g}(\mathbb{Z}/2p) = \mathrm{Sp}_{2g}(\mathbb{Z}/2) \oplus \mathrm{Sp}_{2g}(\mathbb{Z}/p)$. By the homomorphism $\rho: B_n \rightarrow \mathrm{Sp}_{2g}(\mathbb{Z}/2p)$ we deduce that $\rho(B_n[2] \cap B_n[p])$ is trivial. Hence $B_n[2] \cap B_n[p] < B_n[2p]$.
\end{proof}

Now we can prove the main theorem of the section.

\begin{proof}[Proof of Theorem 6.1]
Denote by $s_i$ the transposition $i,i+1$, that is, the generators of $S_n$. We have the following presentation.
\[ S_n = \left< s_1,...,s_{n-1} \: | \: s^2_i=1, s_i s_{i+1} s_i=s_{i+1}s_is_{i+1}, s_i s_j = s_j s_i \: \mathrm{when} \: |i-j|>1  \right>. \]

Consider the natural epimorphism $\tau:B_n \rightarrow S_n$ defined by $\tau(\sigma_i)=s_i$. Fix a prime number $p>2$; then the restriction $\tau:B_n[p] \rightarrow S_n$ is a surjective homomorphism as well. Indeed, we have that $\tau(\sigma^p_i)=s^p_i = s_i$, and for any other generator $g \in B_n[p]$ we have $\tau(g)=1$. Finally, $\mathrm{ker}(\tau) = B_n[2] \cap B_n[p] = B_n[2p]$ by Lemma \ref{evenint}.
\end{proof}

\bibliographystyle{plain}
\bibliography{congruence_subgroups}
\bigskip

Charalampos Stylianakis, department of Mathematics \& Statistics, University of Glasgow, Glasgow, G12 8QW, UK.\\
\textit{E-mail address:} \texttt{c.stylianakis.1@research.gla.ac.uk}


\end{document}